\documentclass[12pt, a4paper]{article}

\renewcommand\footnotemark{}

\usepackage{latexsym,amsmath,enumerate,amssymb,amsbsy,amsthm,color}

\theoremstyle{plain}
\newtheorem{Theorem}{Theorem}[section]
\newtheorem{Proposition}[Theorem]{Proposition}
\newtheorem{Lemma}[Theorem]{Lemma}
\newtheorem{Corollary}[Theorem]{Corollary}

\theoremstyle{definition}

\theoremstyle{remark}

\newtheorem{Example}[Theorem]{Example}

\begin{document}

  \title{Determinants  of Binomial-Related Circulant Matrices}
  
  \date{}
  
  \author{Trairat Jantaramas, Somphong Jitman and  Pornpan Kaewsaard}

  \thanks{T. Jantaramas is  with the Department of Mathematics, Faculty of Science,
      Silpakorn University  Nakhon Pathom 73000,  Thailand.}
  
  \thanks{S. Jitman (Corresponding Author)  is  with the Department of Mathematics, Faculty of Science,
      Silpakorn University  Nakhon Pathom 73000,  Thailand (sjitman@gmail.com).}
  
  \thanks{P.  Kaewsaard is  with the Department of Mathematics, Faculty of Science,
      Silpakorn University  Nakhon Pathom 73000,  Thailand.}

 \maketitle
\begin{abstract} 
    
    Due to their rich algebraic structures and various applications, circulant matrices have been  of interest and continuously studied. 
    In this paper, the notions of Binomial-related matrices have been introduced. Such matrices are  circulant matrices  whose the first row is the coefficients of $(x+zy)^n$, where $z$ is a complex number of norm $1$ and $n$ is a positive integer. 
    In the case where $z\in \{1,-1,i,-i\}$, the explicit formula for the determinant of such matrices are completely determined. Known results on the determinants of  binomial circulant matrices can be viewed as special cases. 
    Finally, some open problems are discussed.

    {\bf 2010 Mathematics Subject Classification}:   11B25, 15A33, 15B05
    
    {\bf Keywords} :   Determinants, Circulant matrices, Binomial coefficients, Euler's formula.
    
\end{abstract}

\section{Introduction}
A complex circulant matrix is a special kind of Toeplitz matrix    whose rows are composed of right cyclically shifted versions of a   list $\boldsymbol{a}=(a_0,a_1,\dots,a_{n-1})\in \mathbb{C}^n$.  Precisely, a complex circulant matrix  is of the form \begin{align*}  \left[\begin{array}{ccccc}
a_0&a_1&\dots&a_{n-2}&a_{n-1}\\
a_{n-1}&a_0&\dots&a_{n-3}&a_{n-2}\\
%a_{n-2}&a_{n-1}&a_0&\dots&a_{n-4}&a_{n-3}\\
\vdots&   \vdots&      \ddots&   \vdots&   \vdots\\ 
a_2&a_3&\dots&a_{0}&a_{1}\\
a_1&a_2&\dots&a_{n-1}&a_{0}\\
\end{array}%
\right]\end{align*} 
for some vector $\boldsymbol{a}=(a_0,a_1,\dots,a_{n-1})\in \mathbb{C}^n$. Such  matrices have extensively been  studied since their
first appearance in the    paper by Catalan  in 1946 (see \cite{c1846}).   In 1994,  the algebraic structures, properties and some applications of circulant matrices  have been summarized in  the book ``Circulant Matrices''    in \cite{D1994}.     These matrices are interesting due to their rich algebraic structures and various applications  (see \cite{D1994},  \cite{GG2009},  \cite{KS2012},     \cite{r1},  and references therein). 
Circulant  matrices have been applied to various disciplines such as such as  image processing, communications, signal processing, networked systems and coding theory.  Circulant matrices can be diagonalized by a discrete Fourier transform, and hence linear equations that contain them may be quickly solved using a fast Fourier transform.   

Determinants   are  known for their applications in matrix theory and linear algebra, e.g., 
determining the area of a triangle via Heron's formula  in \cite{K2004},  solving linear systems using Cramer's rule in \cite{CK2010}, and determining the singularity of a  matrix. 	 Therefore, properties matrices and  determinants of matrices  have been extensively studied  (see \cite{CK2010}, \cite{MMMPSS2008}, and references therein).  The determinants of  circulant matrices have been studied  (see, for example, \cite{D1994}, \cite{KJ2015}, \cite{r6}, and \cite{r5}).
In \cite{R2013}, the determinants of $n\times n$ circulant matrices  whose first row  consists of the coefficients in the Binomial expansion of $(x+y)^{n-1}$ have been   completely  determined.  

In this paper, we focus on  a more general set up. Precisely,  we study  $n\times n$ circulant matrices  whose first row  consists of the coefficients in the  expansion of $(x+zy)^{n-1}$, where $z$ is a nonzero complex number and $n$ is a positive integer. Such matrices will be referred  to as Binomial-related circulant matrices.  Here, the determinants of Binomial-related circulant matrices  are completely determined in the cases where $z\in \{1,-1,i,-i\}$.  The determinant of  the $1\times 1$  Binomial-related circulant matrix  is always $1$.  In this paper, $n$ is assumed to be a positive integer greater than $1$.

The paper is organized as follows. In Section 2,  some basic results on matrices and trigonometric identities are recalled and proved.  The determinants of   $n\times n$ Binomial-related right circulant matrices   are studied in Section 3. In  Section 4, the analogous results for the determinants of   $n\times n$ Binomial-related left circulant matrices are given. Some remarks and open problems are discussed in Section 5.

%%%%%%%%%%%%%%%%%%%%%%%%%%%%%%%
\section{Preliminaries}
In this section, some  basic results on circulant matrices are recalled   together with introduction to the concept of Binomial-related circulant matrices.      Subsequently, some trigonometric identities required in the proofs of the main results are discussed.

\subsection{Left and Right Circulant  Matrices}

Given a positive integer $n$, denote by $M_n(\mathbb{C})$ the set of all $n\times n$ complex matrices. A matrix  $A \in M_n(\mathbb{C})$    is called  a  {\em right circulant matrix}   if  each row of $A$  is   rotated  one element to the right relative to the preceding row.   Precisely,  a complex  right circulant matrix is of the form 

\begin{align*}   %
\left[\begin{array}{ccccc}
a_0&a_1&\dots&a_{n-2}&a_{n-1}\\
a_{n-1}&a_0&\dots&a_{n-3}&a_{n-2}\\
%a_{n-2}&a_{n-1}&a_0&\dots&a_{n-4}&a_{n-3}\\
\vdots&   \vdots&      \ddots&   \vdots&   \vdots\\ 
a_2&a_3&\dots&a_{0}&a_{1}\\
a_1&a_2&\dots&a_{n-1}&a_{0}\\
\end{array}%
\right]=:{\rm rcir}( \boldsymbol{a}),\end{align*} 
where  $\boldsymbol{a}=(a_0,a_1,\dots,a_{n-1})\in \mathbb{C}^n$.

The eigenvalues and the determinants of $n \times n$  circulant matrices have been determined  in terms of the $n$th roots of unity and the elements in its first row (see \cite{D1994}, \cite{KJ2015}, and \cite{r5}).

\begin{Lemma}\label{1.3}
    Let $\boldsymbol{a} = (a_{0},a_{1},\cdots,a_{n-1})\in \mathbb{C}^n$. Then  the  eigenvalues of  ${\rm rcir}({\boldsymbol{a}})$ are  of the form 
    \[\lambda_m = \sum_{k=0}^{n-1} a_ke^\frac{2km\pi i}{n} \]  for all  $m \in \{0,1,\cdots,n-1\}$. 
\end{Lemma}

\begin{Lemma}  \label{2.5}
    Let $A \in M_n(\mathbb{C}) $.  If the eigenvalues of  $A$ are  $\lambda_0,\lambda_1,\dots.,\lambda_{n-1}$, then \[\det(A)=\prod_{i=0}^{n-1} \lambda_i.\]
\end{Lemma}

Combining Lemmas \ref{1.3} and \ref{2.5}, the determinants of circulant matrices can be   easily computed. 

In a similar fashion, a  matrix  $A \in M_n(\mathbb{C})$    is called  a {\em left circulant matrix}  if 
each row of $A$  is   rotated  one element to the left relative to the preceding row.   Precisely,  a complex left circulant matrix is of the form 
\[\left[ {\begin{array}{ccccc}
    a_0 & a_1 & \dots & a_{n-2} & a_{n-1}  \\
    a_1 & a_2 & \dots & a_{n-1} & a_0\\
    %	c_2 & c_3 & \dots & c_0 & c_1\\
    \vdots & \vdots & \ddots & \vdots &\vdots\\
    a_{n-2} & a_{n-1} & \dots & a_{n-4} & a_{n-3}\\
    a_{n-1} & a_0 & \dots & a_{n-3} & a_{n-2}
    \end{array} } \right]=:{\rm lcir }(\boldsymbol{a}) , \]
where  $\boldsymbol{a}=(a_0 ,a_1 , \dots ,  a_{n-1})\in \mathbb{C}^n$. 
 
The determinant of an $n\times n$ left circulant matrix can be determined in terms of the  determinant of a  right circulant matrix and $n$ as follows.

\begin{Lemma}[{\cite{rg}}] \label{col:1}
    Let  $\boldsymbol{a}  \in \mathbb{C}^n$. 
    Then 
    \[ {\rm lcir}(\boldsymbol{a}) = H {\rm rcir}(\boldsymbol{a})  ,\]
    where 
    $	H =  \left[ {\begin{array}{cc}
        1 & O_1 \\
        O^T_1 & z \tilde{I}_{(n-1)}\\
        \end{array} } \right],$
    $\tilde{I}_{(n-1)} = {\rm adiag(1,1,\dots,1)_{(n-1)\times (n-1)}  }
    $ and $O_1= (\displaystyle\underbrace{0,0,\dots,0}_{n-1 \text{copies}})$.
\end{Lemma}

\begin{Corollary} \label{cor-left}
     Let  $\boldsymbol{a}  \in \mathbb{C}^n$. 
    Then 
    \[\det( {\rm lcir}(\boldsymbol{a}) )= (-1)^{\lfloor \frac{n-1}{2}\rfloor} \det({\rm rcir}(\boldsymbol{a}) ),\]
    where $\lfloor x \rfloor$ denotes the greatest integer less than or equal to $x$ for all real  numbers $x$.
\end{Corollary}
\begin{proof}
    From Lemma \ref{col:1},  we note that  $\det( {\rm lcir}(\boldsymbol{a}) )=  \det(H) \det({\rm rcir}(\boldsymbol{a}) )$.  By permuting the $i$th and $(n+1-i)$th rows of $H$ for all $i=1,2,\dots, {\lfloor \frac{n-1}{2}\rfloor} $,   the resulting  matrix is $I_n$. It follows that $\det(H)=  (-1)^{\lfloor \frac{n-1}{2}\rfloor} $ and hence  $\det( {\rm lcir}(\boldsymbol{a}) )= (-1)^{\lfloor \frac{n-1}{2}\rfloor} \det({\rm rcir}(\boldsymbol{a}) )$ as desired. 
\end{proof}

\subsection{Left and Right Binomial-Related Circulant  Matrices}

In \cite{R2013},  $n\times n$ left and right circulant matrices whose first row is $\left(  {n-1 \choose 0},{n-1 \choose 1}, \dots, {n-1 \choose n-1}\right)$, the  coefficients in the Binomial expansion of $(x+y)^{n-1}$ have been introduced and the determinants of such matrices have been completely determined.

Here, we focus on a more general  set up.  For a positive integer $n$ and a nonzero complex number $z$,  let  $\boldsymbol{c}_n (z)$ denote the vector of the coefficients in the expansion of  $(x+zy)^{n-1}$. Precisely,  we have $\boldsymbol{c}_n (z)=\left(  {n-1 \choose 0}z^0,{n-1 \choose 1}z^1, \dots, {n-1 \choose n-1}z^{n-1}\right) $.  In this paper,  right and left Binomial-related circulant matrices $  {\rm rcir}(\boldsymbol{c}_n (z) )$ and $ {\rm lcir}(\boldsymbol{c}_n(z)) $ are studied. 

\begin{Example}  Consider  $(x+iy)^4=x^4+4ix^3y-6x^2y^2-4ixy^3+y^4  $. Then we have 
    
 {\small      \[
    {\rm rcir}(\boldsymbol{c}_5(i)) =\left[\begin{array}{ccccc}
   1 &4i&-6&-4i&1\\
   1&1 &4i&-6&-4i\\
  -4i&1&1 &4i&-6\\
 -6 &-4i&1&1 &4i\\
   4i&  -6 &-4i&1&1 \\
    \end{array}\right] \text{~~and~~}
    {\rm lcir}(\boldsymbol{c}_5(i)) =\left[\begin{array}{ccccc}
   1 &4i&-6&-4i&1\\
    4i&-6&-4i&1&1\\
   -6&-4i&1&1&4i\\
  -4i&1&1&4i&-6\\
   1&4i&-6&-4i\\
    \end{array}\right].\]
}
    
\end{Example}

The eigenvalues of   can be determined in the following lemma.
\begin{Lemma} \label{Eigen}
    Let $n\geq 2$ be a positive integer and let  $z$ be a nonzero complex number. Then the eigenvalues of ${\rm rcir}({\boldsymbol{c}_n(z)})$  are of the form  \[\lambda_m=(1+ze^\frac{2m\pi i}{n})^{n-1}\] for all $m = 0,1,2,\dots,n-1$.
\end{Lemma}
\begin{proof} Let    $m\in  \{ 0,1,2,\dots,n-1\}$.
    	From $\boldsymbol{c}_n (z)=\left(  {n-1 \choose 0}z^0,{n-1 \choose 1}z^1, \dots, {n-1 \choose n-1}z^{n-1}\right) $  and Lemma \ref{1.3},  we have 
    \begin{equation}\label{2.2.1}
    \lambda_m  
    = \sum_{k=0}^{n-1} \begin{pmatrix}	n-1\\k \end{pmatrix}z^k (e^\frac{2m \pi i}{n})^k.
    \end{equation}
    Substituting $y=e^\frac{2m\pi i}{n}$  in
    \begin{equation*}
    (1+zy)^{n-1} = \sum_{k=0}^{n-1}\begin{pmatrix}	n-1\\k \end{pmatrix}  z^k y^k,
    \end{equation*}
  we have
        \begin{equation}\label{2.2.3}
     \sum_{k=0}^{n-1} \begin{pmatrix}	n-1\\k \end{pmatrix}z^k(e^\frac{2m\pi i}{n})^k
    = (1+ze^\frac{2m\pi i}{n})^{n-1}
    \end{equation}
   From  \eqref{2.2.1}  and   \eqref{2.2.3}, 
   it follows that  $
    \lambda_m =  (1+ze^\frac{2m\pi i}{n})^{n-1}$.
\end{proof}

For each $A=[a_{ij}]_{n\times n}\in M_n(\mathbb{C})$, we write $\overline{A}=[\overline{a_{ij}}]_{n\times n}$. We have the following  relation  for the determinants of  Binomial-related circulant matrices. 

\begin{Lemma}\label{conj} Let $n\geq 2$  be a positive integer and let $z$ be a nonzero complex number. Then
    \[  \det({\rm rcir}(\boldsymbol{c}_n (\overline{z}) ))= \overline{\det({\rm rcir}(\boldsymbol{c}_n (z) ))}\text{~~~~and~~~~}  \det({\rm lcir}(\boldsymbol{c}_n (\overline{z}) ))= \overline{\det({\rm lcir}(\boldsymbol{c}_n (z) ))}.\]
\end{Lemma}
\begin{proof} It is not difficult to see that 
   $  {\rm rcir}(\boldsymbol{c}_n (\overline{z}) )= \overline{{\rm rcir}(\boldsymbol{c}_n (z) )}\text{~~~~and~~~~}  {\rm lcir}(\boldsymbol{c}_n (\overline{z}) )= \overline{{\rm lcir}(\boldsymbol{c}_n (z) )}$. Since $\det(\overline{A})=\overline{\det(A)} $ for all $A\in M_n(\mathbb{C})$, the result follows.
\end{proof}

The  determinants  of right and left Binomial-related circulant matrices $  {\rm rcir}(\boldsymbol{c}_n (z) )$ and $ {\rm lcir}(\boldsymbol{c}_n (z)) $ with  $z\in \{1,-1,i,-i\}$ will be determined in Sections 3 and 4.

\subsection{Trigonometric Identities}

In order to determine the determinants of  Binomial-related circulant matrices in the next section,  the following trigonometric identities are required. 

\begin{Lemma} \label{lemEi1}  Let $k$ be a   positive integer. Then 
    \[\prod_{m=1}^{k} \cos \frac{2m\pi}{2k+1} =(-1)^k \prod_{m=1}^{k} \cos \frac{(2m-1)\pi}{2k+1}.\]
\end{Lemma}
\begin{proof}  Since   $\{1,2,\dots, k\} =\{k-m+1\mid k\in \{ 1,2,\dots, k\}\}$, it follows that 
    \begin{align}
    \label{eq1}
    \prod_{m=1}^k \cos \frac{2m\pi}{2k+1} = \prod_{m=1}^k \cos \frac{2(k-m+1)\pi}{2k+1}.
    \end{align}
    For each  $m\in \{1,2,3,\dots, k\}$, we have 
    \begin{align*}  \cos \frac{(2m-1)\pi}{2k+1}&=  \cos \frac{-(2m-1)\pi}{2k+1}  \\
    &=-\cos \frac{(2k+1-(2m-1))\pi}{2k+1}    \\
    &=   -\cos \frac{2(k-m+1)\pi}{2k+1} .
    \end{align*}
    Hence, 
    \begin{align*}  \prod_{m=1}^k   \cos \frac{(2m-1)\pi}{2k+1} =(-1)^k \prod_{m=1}^k   \cos \frac{2(k-m+1)\pi}{2k+1} .
    \end{align*}
    Together with \eqref{eq1}, we have 
    \[\prod_{m=1}^{k} \cos \frac{2m\pi}{2k+1} =(-1)^k \prod_{m=1}^{k} \cos \frac{(2m-1)\pi}{2k+1}\]
    as desired.
\end{proof}

\begin{Lemma}[{\cite[Equation 4.12]{Trig}}] \label{l2.9} Let $k$ be a positive integer. Then 
    \[\prod_{m=1}^{k} \cos \frac{m\pi}{k+1} =  \frac{\sin \frac{(k+1)\pi}{2}}{2^{k}}.\]
\end{Lemma}

\begin{Lemma} \label{l2.10} Let $k$ be a positive integer. Then 
    \[\left(\prod_{m=1}^{k} \cos \frac{2m\pi}{2k+1} \right)^2= \left(  \frac{1}{4} \right)^k.\]
\end{Lemma}
\begin{proof} By Lemma \ref{lemEi1},   it can be deduced that 
    \begin{align*} (-1)^k\left(\prod_{m=1}^{k} \cos \frac{2m\pi}{2k+1} \right)^2 &=\prod_{m=1}^{k} \cos \frac{(2m-1)\pi}{2k+1} \prod_{m=1}^{k} \cos \frac{2m\pi}{2k+1} \\
    &=\prod_{m=1}^{2k} \cos \frac{m\pi}{2k+1}\\
    &=   \frac{\sin\frac{(2k+1)\pi}{2}}{2^{2k}}  ~~~~~~~~~~~~~~~~~~~~~~\text{by Lemma \ref{l2.9},}\\
    &=\left(-\frac{1}{4}\right)^k.
    \end{align*}
    Hence,   we have \[\left(\prod_{m=1}^{k} \cos \frac{2m\pi}{2k+1} \right)^2  =\left(\frac{1}{4}\right)^k\]
    as desired.
\end{proof}

\section{Determinants  of  Binomial-Related Right Circulant Matrices}
In this section,  we focus on the determinants of   Binomial-related right  circulant matrices  in the cases where $z\in \{-1,i,-i\}$.  In the case where $z=1$, the result has been given in \cite{R2013}.

\subsection{Right  Circulant Matrices from the  Coefficients of $(x+y)^{n-1}$ and $(x-y)^{n-1}$}

By substituting $z=1$ and $z=-1$ in Lemma \ref{Eigen}, the next lemma follows. 
\begin{Lemma} \label{1_-1}
    Let $n\geq 2$ be a positive integer. Then the following statements hold.
    \begin{enumerate}
    \item The eigenvalues of ${\rm rcir}({\boldsymbol{c}_n(1)})$  are of the form  $\lambda_m=(1+e^\frac{2m\pi i}{n})^{n-1}$  for all $m = 0,1,2,\dots,n-1$.
    \item  The eigenvalues of ${\rm rcir}({\boldsymbol{c}_n(-1)})$  are of the form  $\lambda_m=(1-e^\frac{2m\pi i}{n})^{n-1}$  for all $m = 0,1,2,\dots,n-1$.
    \end{enumerate}
\end{Lemma}
Based on Lemma \ref{1_-1}, the determinant of ${\rm rcir}({\boldsymbol{c}_n(1)})$ has been given in \cite{R2013}.

\begin{Proposition}[{\cite[Theorem 2.1]{R2013}}] \label{prop1} Let $n\geq 2$ be a positive integer. Then
    	\begin{equation*}
    \det({\rm rcir}({\boldsymbol{c}_n(1)}))= (1+(-1)^{n-1})2^{n-2}.
    \end{equation*}
\end{Proposition}

From Lemma \ref{1_-1}, observe that the eigenvalue $\lambda_0$ of ${\rm rcir}({\boldsymbol{c}_n(-1)})$ is  $\lambda_0=(1-e^0)^{n-1}=0$.  By Lemmas \ref{1.3} and \ref{2.5},     the determinant of  ${\rm rcir}({\boldsymbol{c}_n(-1)})$ follows immediately.
\begin{Proposition}\label{prop-1} Let $n\geq 2$ be a positive integer. Then
    \begin{equation*}
    \det({\rm rcir}({\boldsymbol{c}_n(-1)}))= 0.
    \end{equation*}
\end{Proposition}

\subsection{Right  Circulant Matrices from  the Coefficients of $(x+iy)^{n-1}$}
In this subsection, we focus on the determinant of $\boldsymbol{c}_n (i)$.  The results are given in terms of the residues of $n$ modulo $4$.

  By setting $z=i$ in   Lemma \ref{Eigen},  we have the following lemma.

\begin{Lemma} \label{2.2}
    Let $n\geq 2$ be a positive integer.  Then the eigenvalues of ${\rm rcir}({\boldsymbol{c}_n(i)})$  are of the form  $\lambda_m=(1+ie^\frac{2m\pi i}{n})^{n-1}$  for all $m = 0,1,2,\dots,n-1$.
\end{Lemma}
 
The following properties of the eigenvalues of  ${\rm rcir}({\boldsymbol{c}_n(i)})$   are key to prove the main results.

\begin{Lemma}  \label{l3.5} Let $n\geq 2$ be a positive integer. Then
    \[\lambda_m\lambda_{n-m}= \left( 2i \cos \frac{2m\pi }{n}\right)^{n-1} \]
    for all $1\leq m <n$, where $\lambda_m$ and  $\lambda_{n-m}$ are given  in Lemma \ref{2.2} 
\end{Lemma}
\begin{proof}  
    Let $m$ be an integer such that $1\leq m<n$. Then 
    \begin{align*}
    \lambda_m\lambda_{n-m}&=(1+ie^\frac{2m\pi i}{n})^{n-1}(1+ie^\frac{2(n-m)\pi i}{n})^{n-1}\\
    &=\left((1+ie^\frac{2m\pi i}{n})(1+ie^\frac{2(n-m)\pi i}{n})\right)^{n-1}\\
    &=\left((1+i(e^\frac{2m\pi i}{n}+e^\frac{2(n-m)\pi i}{n}) -1\right)^{n-1}\\
    &=\left(i(e^\frac{2m\pi i}{n}+e^\frac{-2m\pi i}{n}) \right)^{n-1}\\
    &=\left( 2i \cos \frac{2m\pi }{n}\right)^{n-1}
    \end{align*}
   by the Euler's formula.
\end{proof}

\begin{Lemma} \label{lem:propSpecLam}Let $n\geq 2$ be a  positive integer. Then the following statements hold.
    \begin{enumerate}
        \item If $n$ is odd, then
        \[\lambda_0=\left(2i\right)^{\frac{n-1}{2}}.\]
        \item  If $n$ is even, then 
        \[\lambda_0\lambda_{\frac{n}{2}} = 2^{n-1}.\]
        \item  If $n\equiv 0\,{\rm mod}\,4$, then 
        \[ \lambda_{\frac{n}{4}}  =0.\]
    \end{enumerate} 
\end{Lemma}
\begin{proof}  Assume that $n$ is odd. Then $n-1$ is even and hence 
    \begin{align*}
    \lambda_0&=(1+i)^{n-1}\\
    &=\left((1+i)^2\right)^{\frac{n-1}{2}}\\
    &=\left(1+2i-1\right)^{\frac{n-1}{2}}\\
    &=\left(2i\right)^{\frac{n-1}{2}}.
    \end{align*}
   Assume that $n$ is even. Then
    \begin{align*} \lambda_0\lambda_{\frac{n}{2}}&=(1+i)^{n-1} (1+ie^ {\pi i})^{n-1} \\
    &= (1+i)^{n-1} (1-i)^{n-1} \\
    &=2^{n-1}.
    \end{align*}
    Assume that $n\equiv 0\,{\rm mod }\, 4$. Then
    \begin{align*}  \lambda_{\frac{n}{4}}&=  (1+ie^ {\frac{\pi i}{2}})^{n-1} \\
    &= (1+i^2)^{n-1}  \\
    &=0.
    \end{align*}
    The lemma is proved.
\end{proof}

Next, the determinant of ${\rm rcir}({\boldsymbol{c}_n(i)})$ is determined in the following four cases.

\begin{Proposition}  If $n \equiv 0\,{\rm mod}\, 4$, then 
    \[\det({\rm rcir}({\boldsymbol{c}_n(i)}))=0.\]
\end{Proposition}

\begin{proof}
    Assume that   $n\geq 2$  is a positive integer such that  $n \equiv 0\,{\rm mod}\, 4$.  By Lemma \ref{lem:propSpecLam}, we have $\lambda_{\frac{n}{4}}=0$ and hence
    \[\det({\rm rcir}({\boldsymbol{c}_n(i)}))=\prod_{m=0}^{n-1} \lambda_m=0\] 
    by Lemma \ref{2.5}. 
\end{proof}

\begin{Proposition}  If $n \equiv 1\,{\rm mod}\, 4$, then 
    \[\det({\rm rcir}({\boldsymbol{c}_n(i)}))=(2i)^{\frac{n-1}{2}} .\]
\end{Proposition}
\begin{proof} Let $n\geq 2$ be a positive integer such that $n \equiv 1\,{\rm mod}\, 4$.  Then $n=4a+1$ for some positive integer $a$. 
  By Lemmas \ref{1.3} and \ref{2.5}, we have 
  
    \begin{align*}
    \det ({\rm rcir}({\boldsymbol{c}_n(i)}))&= \prod_{m=0}^{n-1} \lambda_m\\
    &=\lambda_0 \prod_{m=1}^{\frac{n-1}{2}} \lambda_m\lambda_{n-m}\\
    &= \left(2i\right)^{\frac{n-1}{2}}  \prod_{m=1}^{ \frac{n-1}{2}} \left( 2i \cos \frac{2m\pi }{n}\right)^{n-1}~~~~\text{ by Lemmas  \ref{l3.5} and \ref{lem:propSpecLam},}\\
    &= \left(2i\right)^{\frac{n-1}{2}}  \left(  \prod_{m=1}^{ 2a} 2i \cos \frac{2m\pi }{4a+1}\right)^{4a}\\
    &= \left(2i\right)^{\frac{n-1}{2}}  \left(  (2i)^{2a}\prod_{m=1}^{ 2a}  \cos \frac{2m\pi }{4a+1}\right)^{4a}\\
    &= \left(2i\right)^{\frac{n-1}{2}}  \left(  (2i)^{4a} \left(\prod_{m=1}^{ 2a}  \cos \frac{2m\pi }{4a+1}\right)^2\right)^{2a}\\
    &= \left(2i\right)^{\frac{n-1}{2}}  \left(  2^{4a}  \left(\frac{1}{4} \right)^{2a}\right)^{2a} ~~~~~~~~~~~~~\text{ by Lemma \ref{l2.10},}\\
    &= \left(2i\right)^{\frac{n-1}{2}} . 
    \end{align*}
    Therefore, we have $\det({\rm rcir}({\boldsymbol{c}_n(i)}))=(2i)^{\frac{n-1}{2}} $.
\end{proof}

\begin{Proposition}  If $n \equiv 2\,{\rm mod}\, 4$, then 
    \[\det({\rm rcir}({\boldsymbol{c}_n(i)}))=2^{n-1}.\]
\end{Proposition}

\begin{proof}
    Let $n\geq 2$ be a positive integer such that $n \equiv 2\,{\rm mod}\, 4$.  Then $n=2(2a+1)$ for some positive integer $a$.  By Lemmas \ref{1.3} and \ref{2.5},  it can be deduced that 
    \begin{align*}
    \det ({\rm rcir}({\boldsymbol{c}_n(i)}))&= \prod_{m=0}^{n-1} \lambda_m\\
    &=\lambda_0 \lambda_{\frac{n}{2}}  \prod_{m=1}^{\frac{n}{2}-1} \lambda_m\lambda_{n-m}\\
    &=  2^{n-1}\prod_{m=1}^{ \frac{n}{2}-1} \left( 2i \cos \frac{2m\pi }{n}\right)^{n-1} ~~~~~~~~~~~~~\text{ by Lemmas  \ref{l3.5} and \ref{lem:propSpecLam},}\\
    &= 2^{n-1} \left(  \prod_{m=1}^{ 2a} 2i \cos \frac{2m\pi }{4a+2}\right)^{4a+1} \\
    &= 2^{n-1} \left(  (2i )^{2a}\prod_{m=1}^{ 2a} \cos \frac{m\pi }{2a+1}\right)^{4a+1} \\
    &= 2^{n-1} \left(  (2)^{2a}(-1)^a  \frac{\sin\frac{(2a+1)\pi}{2}}{2^{2a}}    \right)^{4a+1}   ~~~ ~\text{ by Lemma \ref{l2.10},}\\
    &= 2^{n-1} \left(  (2)^{2a}(-1)^a  \frac{(-1)^{a}}{2^{2a}}    \right)^{4a+1} \\
    &=   2^{n-1}.
    \end{align*} Hence, $\det({\rm rcir}({\boldsymbol{c}_n(i)}))=2^{n-1}$ as desired.
\end{proof}

\begin{Proposition}  If $n \equiv 3\,{\rm mod}\, 4$, then 
    \[\det({\rm rcir}({\boldsymbol{c}_n(i)}))=- (2i)^{\frac{n-1}{2}} .\]
\end{Proposition}
\begin{proof} Let $n\geq 2$ be a positive integer such that $n \equiv 3\,{\rm mod}\, 4$.  Then $n=4a+3$ for some positive integer $a$. By Lemmas \ref{1.3} and \ref{2.5}, we have 
    \begin{align*}
    \det ({\rm rcir}({\boldsymbol{c}_n(i)}))&= \prod_{m=0}^{n-1} \lambda_m\\
    &=\lambda_0 \prod_{m=1}^{\frac{n-1}{2}} \lambda_m\lambda_{n-m}\\
    &= \left(2i\right)^{\frac{n-1}{2}}  \prod_{m=1}^{ \frac{n-1}{2}} \left( 2i \cos \frac{2m\pi }{n}\right)^{n-1} ~~~~\text{ by Lemmas  \ref{l3.5} and \ref{lem:propSpecLam},}\\
    &= \left(2i\right)^{\frac{n-1}{2}}  \left(  \prod_{m=1}^{ 2a+1} 2i \cos \frac{2m\pi }{4a+3}\right)^{4a+2}\\
    &= \left(2i\right)^{\frac{n-1}{2}}  \left(  (2i)^{2a+1}\prod_{m=1}^{ 2a+1}  \cos \frac{2m\pi }{4a+3}\right)^{4a+2}\\
    &= \left(2i\right)^{\frac{n-1}{2}}  \left(  (2i)^{4a+2} \left(\prod_{m=1}^{ 2a+1}  \cos \frac{2m\pi }{4a+3}\right)^2\right)^{2a+1}\\
    &= \left(2i\right)^{\frac{n-1}{2}}  \left(  -2^{4a+2}  \left(\frac{1}{4} \right)^{2a+1}\right)^{2a+1}~~~~~~~ ~~\text{ by Lemma \ref{l2.10},}\\
    &= -\left(2i\right)^{\frac{n-1}{2}} . 
    \end{align*}
    Therefore,   
   $\det({\rm rcir}({\boldsymbol{c}_n(i)}))=- (2i)^{\frac{n-1}{2}}$ for all positive integers $n$ such that $n \equiv 3\,{\rm mod}\, 4$.
\end{proof}

The results can be summarized as follows.
\begin{Theorem}\label{thm_i} Let $n\geq 2$ be a positive integer. Then 
    \[\det({\rm rcir}({\boldsymbol{c}_n(i)}))= \begin{cases}
    0   & \text{ if } n\equiv 0\,{\rm mod}\, 4\\
    (2i)^{\frac{n-1}{2}}   & \text{ if }n\equiv 1\,{\rm mod}\, 4\\
    2^{n-1}   &  \text{ if }n\equiv 2\,{\rm mod}\, 4\\
    -(2i)^{\frac{n-1}{2}}     & \text{ if }n\equiv 3\,{\rm mod}\, 4.
    \end{cases}\]
\end{Theorem}

\subsection{Right  Circulant Matrices from Coefficients of $(x-iy)^{n-1}$}
In this subsection,  we focus on the determinant of ${\rm rcir}({\boldsymbol{c}_n(-i)})$. The formula for  ${\rm rcir}({\boldsymbol{c}_n(-i)})$ can be given in based on the determinant of ${\rm rcir}({\boldsymbol{c}_n(i)})$ given in Subsection 3.2.

By setting $z=-i$ in   Lemma \ref{Eigen}, the eigenvalues of ${\rm rcir}({\boldsymbol{c}_n(-i)})$.

\begin{Lemma} 
	Let $n\geq 2$ be a positive integer.  Then the eigenvalues of ${\rm rcir}({\boldsymbol{c}_n(-i)})$  are of the form  $\lambda_m=(1-ie^\frac{2m\pi i}{n})^{n-1}$  for all $m = 0,1,2,\dots,n-1$.
\end{Lemma}
Using the analysis as in Subsection 2.2, the determinant of  ${\rm rcir}({\boldsymbol{c}_n(-i)})$ can be determined.

Alternatively, we have $-i=\overline{i}$. From Lemma \ref{conj}, we have  \[\det({\rm rcir}({\boldsymbol{c}_n(-i)}))= \det({\rm rcir}({\boldsymbol{c}_n(\overline{i})})) =\overline{ \det({\rm rcir}({\boldsymbol{c}_n({i})}))} .\]
Based on Theorem \ref{thm_i}, the formula for $\det({\rm rcir}({\boldsymbol{c}_n(-i)}))$ can be derived in the following theorem.
 
\begin{Theorem} \label{thm_-i}Let $n\geq 2$ be a positive integer. Then 
	\[\det({\rm rcir}({\boldsymbol{c}_n(-i)}))= \begin{cases}
	0   & \text{ if } n\equiv 0\,{\rm mod}\, 4\\
	(2i)^{\frac{n-1}{2}}   & \text{ if }n\equiv 1\,{\rm mod}\, 4 \text{ or }n\equiv 3\,{\rm mod}\, 4.\\
	2^{n-1}   &  \text{ if }n\equiv 2\,{\rm mod}\, 4 .
	\end{cases}\]
\end{Theorem}
\begin{proof}
	From Theorem \ref{thm_i}, we have 
	  \begin{align*}
	  \det({\rm rcir}({\boldsymbol{c}_n(-i)}))  
	  =\overline{\det({\rm rcir}({\boldsymbol{c}_n({i})}))}
	   = \begin{cases}
	0   & \text{ if } n\equiv 0\,{\rm mod}\, 4\\
	(-2i)^{\frac{n-1}{2}}   & \text{ if }n\equiv 1\,{\rm mod}\, 4\\
	2^{n-1}   &  \text{ if }n\equiv 2\,{\rm mod}\, 4\\
	-(-2i)^{\frac{n-1}{2}}     & \text{ if }n\equiv 3\,{\rm mod}\, 4.
	\end{cases}
	\end{align*}
	If $n\equiv 1\,{\rm mod}\, 4$, then $\frac{n-1}{2}$ is even and hence $(-1)^{\frac{n-1}{2}}=1 $. For  $n\equiv 3\,{\rm mod}\, 4$, then $\frac{n-1}{2}$ is odd and   $(-1)^{\frac{n-1}{2}}=-1 $. The result is therefore follows.
\end{proof}

\section{Determinants  of  Binomial-Related Left Circulant Matrices}

In this section, a brief summary on the determinant  of $ {\rm lcir}({\boldsymbol{c}_n(z)})$  is given based on $\det({\rm rcir}({\boldsymbol{c}_n(z)}))$ determined in Section 3  and   Corollary \ref{cor-left}. 

 From Corollary \ref{cor-left},  we have  \[\det( {\rm lcir}(\boldsymbol{a}) )= (-1)^{\lfloor \frac{n-1}{2}\rfloor} \det({\rm rcir}(\boldsymbol{a}) )\]
 for all $\boldsymbol{a}\in \mathbb{C}^n$.

Note that $    (-1)^{\lfloor \frac{n-1}{2}\rfloor} =1 $  if and only if $ n\equiv 1\,{\rm mod}\, 4 $ or  $ n\equiv 2\,{\rm mod}\, 4 $;  and   $    (-1)^{\lfloor \frac{n-1}{2}\rfloor} =-1 $  if and only if  $ n\equiv 0\,{\rm mod}\, 4 $  or  $ n\equiv 3\,{\rm mod}\, 4 $.  Together with Proposition \ref{prop1}, Proposition \ref{prop-1}, Theorem  \ref{thm_i} and Theorem \ref{thm_-i}, the following results concerning the determinants of Binomial-related left circulant matrices can be concluded.

 \begin{Proposition}[{\cite[Theorem 2.3]{R2013}}]  Let $n\geq 2$ be a positive integer. Then
 	\begin{equation*}
 	\det({\rm lcir}({\boldsymbol{c}_n(1)}))=(-1)^{\lfloor \frac{n-1}{2}\rfloor}  (1+(-1)^{n-1})2^{n-2}.
 	\end{equation*}
 \end{Proposition}

 \begin{Proposition} Let $n\geq 2$ be a positive integer. Then
 	\begin{equation*}
 	\det({\rm lcir}({\boldsymbol{c}_n(-1)}))= 0.
 	\end{equation*}
 \end{Proposition}

 \begin{Theorem} \label{t4.3} Let $n\geq 2$ be a positive integer. Then 
 	\[\det({\rm lcir}({\boldsymbol{c}_n(i)}))= \begin{cases}
 	0   & \text{ if } n\equiv 0\,{\rm mod}\, 4\\
 	(2i)^{\frac{n-1}{2}}   & \text{ if }n\equiv 1\,{\rm mod}\, 4 \text{ or }n\equiv 3\,{\rm mod}\, 4.\\
 	 2^{n-1}   &  \text{ if }n\equiv 2\,{\rm mod}\, 4 
 	\end{cases}\]
 \end{Theorem}

Observe that  $\det({\rm rcir}({\boldsymbol{c}_n(-i)}))= \det({\rm lcir}({\boldsymbol{c}_n(i)}))$ by Theorems \ref{thm_-i} and \ref{t4.3}.
 
 \begin{Theorem} \label{thm4.4} Let $n\geq 2$ be a positive integer. Then 
 	\[\det({\rm lcir}({\boldsymbol{c}_n(-i)}))= \begin{cases}
 	0   & \text{ if } n\equiv 0\,{\rm mod}\, 4\\
 	(-1)^{  \frac{n-1}{2} } (2i)^{\frac{n-1}{2}}   & \text{ if }n\equiv 1\,{\rm mod}\, 4 \text{ or }n\equiv 3\,{\rm mod}\, 4.\\
 	   2^{n-1}   &  \text{ if }n\equiv 2\,{\rm mod}\, 4 
 	\end{cases}\]
 \end{Theorem}
 Observe that  $\det({\rm rcir}({\boldsymbol{c}_n(i)}))= \det({\rm lcir}({\boldsymbol{c}_n(-i)}))$ by Theorems \ref{thm_i} and \ref{t4.4}.

%\subsection{Left Circulant Matrices from Coefficients of $(x+y)^n$ and $(x-y)^{n-1}$}
%
%\subsection{Left Circulant Matrices from Coefficients of $(x+iy)^n$ and $(x-iy)^{n-1}$}

\section{Conclusion and Remarks} The concept of complex Binomial-related   left and right  circulant matrices has been introduced. Such matrices are complex $n\times n$   circulant matrices  whose first row consists of the coefficients in the expansion of $(x+zy)^{n-1}$. In the case where $z=1$, the determinants of  Binomial  left and right  circulant matrices  have been determined in \cite{R2013}. In this paper,   the eigenvalues and the  determinants of  Binomial-related  left and right  circulant matrices  have been completely determined for $z\in \{-1,i,-1\}$. It is of natural interest to investigate  the eigenvalues and the  determinants of such matrices for other forms of $z$. This issue is remained as an open problem.

\section*{Acknowledgments}
{S. Jitman was supported by the Thailand Research Fund  under Research
    Grant MRG6080012.}

\end{document}